\documentclass[12pt,a4paper,openany]{article}

\usepackage{amsmath,amssymb,amsfonts,amsthm}
\usepackage{graphics}
\usepackage{mathrsfs}
\usepackage[all]{xy}
\usepackage{color}

\newtheorem{thm}{Theorem}[section]
\newtheorem{prop}[thm]{Proposition}
\newtheorem{lemma}[thm]{Lemma}

\newtheorem{cor}[thm]{Corollary}

\newtheorem{conj}[thm]{Conjecture}

\theoremstyle{definition}
\newtheorem{definition}[thm]{Definition}
\newtheorem{rem}[thm]{Remark}

\newtheorem*{ack}{Acknowledgment}

\newcommand{\sq}{\hfill $\square$}
\newcommand{\kah}{K\"{a}hler}
\newcommand{\dol}{\sqrt{-1}\partial \overline{\partial}}

\newcommand{\hol}{H\"{o}lder}
\newcommand{\ma}{Monge-Amp$\grave{{\rm e}}$re}

\makeatletter
\def\address#1#2{\begingroup
\noindent\parbox[t]{16cm}{%
\small{\scshape\ignorespaces#1}\par\vskip1ex
\noindent\small{\itshape E-mail address}%
\/: #2\par\vskip4ex}\hfill%
\endgroup}%
\makeatother

\makeatletter

\@addtoreset{equation}{section}
\makeatother

\title{A conical approximation of constant scalar curvature {\kah} metrics of Poincar\'{e} type}
\author{Takahiro Aoi}
\date{\today}
\begin{document}
\maketitle

\begin{abstract}
Let $(X,L_X)$ be a polarized manifold and $D$ be a smooth hypersurface such that $D \in | L_X |$.
In this paper, we show that if there is no nontrivial holomorphic vector field on $D$ and ${\rm Aut}_0 ((X,L_X); D)$ is trivial, then constant scalar curvature {\kah} metrics of Poincar\'{e} type on $X \setminus D$ can be approximated by constant scalar curvature {\kah} metrics with cone singularities of sufficiently small angle along $D$.
This result implies log K-semistability of $((X,L_X);D)$ with angle 0.
\end{abstract}


\section{Introduction}

In {\kah} geometry, the existence of {\bf constant scalar curvature {\kah} (cscK) metrics} is a fundamental problem.
There are many works on the existence of cscK metrics on the complement of a divisor on a compact {\kah} manifold.
In particular, Auvray have studied cscK (more generally, extremal {\kah}) metrics of Poincar\'{e} type, i.e., with cusp singularities \cite{Au1,Au2,Au3}.
On the other hand, Guenancia \cite{Gu} and Biquard-Guenancia \cite{BG} showed that some complete {\kah}-Einstein metric can be realized as the limit of a sequence of {\kah}-Einstein metrics with cone singularities along a divisor.
As an analogue of Guenancia's result \cite{Gu} for cscK metrics, we prove the following theorem in this paper.

\begin{thm}
\label{cone to cusp}
Let $(X,L_X)$ be an $n$-dimensional polarized manifold with a smooth hypersurface $D \in | L_X |$.
Assume that there is no nontrivial holomorphic vector field on $D$ and ${\rm Aut}_0 ((X,L_X); D)$ is trivial.
If $X \setminus D$ admits a cscK metric $\omega_{0}^{cscK} \in c_1 (L_X)$ of Poincar\'{e} type, then there exists a cscK cone metric $\omega_{\beta}^{cscK} \in c_1(L_X)$ of cone angle $2 \pi \beta$ along $D$ for any $0<\beta \ll 1$.
Moreover, $\omega_{\beta}^{cscK}$ converges to $\omega_{0}^{cscK}$ as $\beta \to 0$, in the sense of topology of $C^{4,\alpha}_{\eta}(X \setminus D)$ for some $-1 \ll \eta <0$.
\end{thm}

Note that we consider the construction of a cscK cone metric in the \textbf{fixed} cohomology class $c_1 (L_X)$ (cf. \cite{Gu}).
Here, $C^{4,\alpha}_{\eta}(X \setminus D)$ denotes the weighted {\hol} space (see \S 2.1 in this paper and \cite{Au2}) and ${\rm Aut}_0 ((X,L_X); D)$ denotes the connected component of the identity in the holomorphic automorphism group preserving $D$.
The hypothesis on ${\rm Aut}_0 ((X,L_X); D)$ is a generic condition since we can find such a divisor generically by replacing $L_X$ with $mL_X$ for sufficiently large $m$ (see Proposition 2.11 in \cite{AHZ}). 
In fact, the standard elliptic theory implies that $\omega_{\beta}^{cscK} \to \omega_{0}^{cscK}$ in $C^{\infty}_{{\rm loc}} (X\setminus D)$ (see \S 3.6 in \cite{Aub}).
In particular, the sequence of metric spaces $(X \setminus D, \omega_{\beta}^{cscK}, p)$ converge to $(X \setminus D, \omega_{0}^{cscK}, p)$ in the sense of the pointed Gromov-Hausdorff topology where $p \in X \setminus D$ is a fixed point.
It is expected that the existence of cscK cone metrics is equivalent to log $K$-stability (\cite{Do2}, see also \cite{AHZ}). 
Zheng \cite{Zhe2} showed that the existence of cscK cone metrics is equivalent to log geodesic stability.
By using this result, we can show that the existence of cscK cone metrics implies ${\rm Aut}_0 ((X,L_X); D)$-uniform log $K$-stability for normal test configurations \cite{AHZ}.
Thus, we immediately obtain the following corollary.
\begin{cor}
Assume that there is no nontrivial holomorphic vector field on $D$ and ${\rm Aut}_0 ((X,L_X); D)$ is trivial.
If $X \setminus D$ admits a cscK metric of Poincar\'{e} type, then $((X,L_X); D)$ is uniformly log $K$-stable with sufficiently small angle $2 \pi \beta$ for normal test configurations.
\end{cor}

In \cite{J.Sun}, J. Sun proved that the existence of a cscK metrics of Poincar\'{e} type (constructed by Calabi ansatz) on the total space of some line bundle over a cscK manifold with negative scalar curvature implies log K-semistability with angle 0.
Since the log Donaldson-Futaki invariant depends linearly on $\beta$ (\cite{Do2}, see also \cite{AHZ}), we have the following result which is a generalization of Sun's result.
\begin{cor}
\label{log K-semi}
Assume that there is no nontrivial holomorphic vector field on $D$ and ${\rm Aut}_0 ((X,L_X); D)$ is trivial.
If $X \setminus D$ admits a cscK metric of Poincar\'{e} type, then $((X,L_X); D)$ is log $K$-semistable with cone angle $0$ for normal test configurations.
\end{cor}
This result is a partial solution of Sz\'{e}kelyhidi's conjecture when $H^0 (D, TD)=0$ and ${\rm Aut}_0 ((X,L_X); D)$ is trivial (see \cite{Sz1, AAS, Se}).
On the other hand, J. Sun and S. Sun conjectured that if $D$ has a cscK metric with nonpositive scalar curvature, then $((X,L_X); D)$ is log $K$-semistable with cone angle $0$ (see Conjecture 1.1 in \cite{SS}).
This conjecture has already been solved when $D$ has a scalar-flat {\kah} metric in \cite{S.Sun} (see also Introduction of \cite{SS}).
Auvray showed that the existence of a cscK metric of Poincar\'{e} type on $X \setminus D$ implies the existence of a cscK metric with negative scalar curvature on the divisor $D \in | L_X |$ (see Remark \ref{negative scalar curvature} in this paper).
So, Corollary \ref{log K-semi} is a solution of Conjecture 1.1 in \cite{SS} under the (stronger) hypothesis that $X \setminus D$ admits a cscK metric of Poincar\'{e} type (see the following diagram).

\begin{flalign*}
\xymatrix{
\mbox{$X \setminus D$ : Poincar\'{e} type cscK } \ar@{->}[r]^{{\rm \hspace{5pt} Auvray's \hspace{3pt} result \hspace{3pt} [Au3]}} \ar@{->}[rd]^{{\rm Corollary \hspace{2pt} \ref{log K-semi}}} \ar@{<.>}[d]^{{\rm {{Conjecture \hspace{3pt} [Sz]}}}} &  \ar@{.>}[d]^{{\rm {Conjecture \hspace{3pt} [SS]}}} \mbox{$D$ : negative cscK}\\
\mbox{$((X,L_X);D)$ : log K-stable for angle 0} \ar@{->}[r] & \mbox{$((X,L_X);D)$ : log K-semistable for angle 0} \\
}
\end{flalign*}
\vspace{8pt}

\begin{rem}
Let $\underline{S}_{D}$ be the average of the scalar curvature of a {\kah} metric in the class $c_1 (L_X |_D)$.
It is known that $((X,L_X);D)$ is log K-unstable with angle $2\pi \beta$ if $\beta < \frac{\underline{S}_{D}}{n (n-1)}$ in \cite{AHZ}.
In our setting: $D \in | L_X |$, Auvray's topological constraint \cite{Au1} tells us that the average of the scalar curvature on $D$ is negative: $\underline{S}_{D}<0$ (see Remark \ref{negative scalar curvature} in this paper).
Thus, Corollary \ref{log K-semi} does not contradict the result in \cite{AHZ}.
\end{rem}

This paper is organized as follows.
In Section 2, we recall the definitions of {\kah} metrics of Poincar\'{e}/cone singularities.
In addition, we define a background {\kah} metric with cone singularities, which converges to a cscK metric of Poincar\'{e} type.
Finally, we consider a fixed point formula on the weighted {\hol} space, which characterizing a cscK metric with cone singularities.
In Section 3, we prove Theorem \ref{cone to cusp}.

\begin{ack}
The author would like to thank Lars Martin Sektnan for helpful comments.
He also would like to thank Yoshinori Hashimoto for helpful comments and many discussions.
\end{ack}


\section{Preliminaries}

In this section, we consider the following.
Firstly, we recall the definitions of {\kah} metrics of Poincar\'{e} type and the weighted {\hol} space by following \cite{Au2}.
Secondly, we recall the definition of {\kah} metrics with cone singularities along a divisor by following \cite{Zhe1}.
Finally, we define a {\kah} metric with cone singularities, which converges locally to a cscK metric of Poincar\'{e} type.
(We will show that its scalar curvature is asymptotic to the average value of the scalar curvature in Lemma \ref{estimate at infinity} later.)
In addition, we consider a fixed point formula which characterizing a cscK cone metric.

\subsection{{\kah} metrics of Poincar\'{e} type and the weighted {\hol} space}

Let $(X, \theta_X)$ be an $n$-dimensional compact {\kah} manifold and $D$ be a smooth hypersurface in $X$.
We recall the definition of {\kah} metrics of Poincar\'{e} type for the simplest case when $D$ is a smooth divisor, by following \cite{Au1,Au2,Au3}.
In this paper, $(U ; z^1,...,z^n)$ denotes a holomorphic coordinate chart such that $U \cap D = \{ z^1 = 0\}$ and $| z^1 | <1$.

\begin{definition}[\cite{Au2}]
A {\kah} metric $\Theta$ on $X\setminus D$ is said to be \textit{of Poincar\'{e} type} if for all holomorphic coordinate chart $(U ; z^1,...,z^n)$, $\Theta$ is quasi-isometric to the standard cusp metric
\begin{equation}
\omega_{cusp} := \frac{\sqrt{-1} d z^1 \wedge d \overline{z^1}}{|z^1|^2 \log^2 |z^1|^2} + \sum_{j=2}^{n} \sqrt{-1} d z^j \wedge d \overline{z^j},
\end{equation}
and its derivatives are bounded at any order with respect to this model metric.

We say that $\Theta$ has class $[\theta_X]$ if it can be written as $\Theta = \theta_X + \dol \varphi$ for some smooth function $\varphi$ on $X \setminus D$ such that $\varphi = O(1 + \log (- \log |z^1|))$ and  its derivatives are bounded at any positive order with respect to the model metric above. \sq
\end{definition}

\begin{rem}
{\kah} metrics of Poincar\'{e} type can be defined for a simple normal crossing divisor $D$ (see \cite{Au1,Au2, Au3}).
\end{rem}

For {\kah} metrics of Poincar\'{e} type, there is a Banach space of functions on $X \setminus D$ defined by Kobayashi \cite{Ko} and Auvray \cite{Au2} by using quasi-coordinates introduced by Cheng-Yau \cite{CY}.
We recall quickly the definition of it for the reader's convenience (for more detail, see \cite{CY, Ko, Au2}).
We consider the standard cusp metric $\omega_{cusp}$ on the punctured disc $\Delta^{*} := \{ 0< |z| < 1 \} \subset \mathbb{C}$.
For $\delta \in (0,1)$, we define the following holomorphic map: $\varphi_\delta : \frac{3}{4}  \Delta \rightarrow \Delta^* , \zeta  \mapsto \exp \left( - \frac{1 + \delta}{1 - \delta} \frac{1 + \zeta}{1 - \zeta}  \right)$.
Here, $\Delta := \{ |z| <1\} \subset \mathbb{C}$.
It is known that $ \bigcup_{\delta \in (0,1)} \varphi_\delta (\frac{3}{4}  \Delta) = \Delta^*$ (see Section 2 in \cite{Ko}).
Directly, we note that the pull back of this {\kah} metric :
\begin{equation}
\varphi_{\delta}^{*} \omega_{cusp} = \frac{\sqrt{-1} d \zeta \wedge d \overline{\zeta}}{(1-|\zeta|^2 )^2},
\end{equation}
is independent of $\delta \in (0,1)$ and is $C^{\infty}$-quasi-isometric to the Euclidean metric on $\frac{3}{4} \Delta$.
By using this holomorphic map $\varphi_\delta$, we define a holomorphic map $\Phi_\delta : \mathcal{P} := \frac{3}{4}  \Delta \times \Delta^{n-1} \to \Delta^* \times \Delta^{n-1} , (\zeta_1, z_2,..., z_n) \mapsto ( \varphi_\delta(\zeta_1), z_2,..., z_n)$.
For $k \in \mathbb{Z}_{\geq 0}, \alpha \in (0,1)$, we define the $C^{k, \alpha}(U \setminus D)$-norm by
\begin{equation}
\label{local norm}
\Vert  f \Vert _{C^{k, \alpha}(U \setminus D)} := \sup_{\delta \in (0,1)} \Vert  \Phi^{*}_{\delta} f \Vert _{C^{k, \alpha}(\mathcal{P})}.
\end{equation}
Here, we have identified $U \setminus D \simeq \Delta^* \times \Delta^{n-1}$.
We take an open subset $U_0 \Subset X \setminus D$ and a finite covering $\{ U_i \}_{i=1}^{N}$ of $D$ such that $U_0 \bigcup \left( \cup_{i=1}^{N} U_i \right) = X$.

\begin{definition}[The {\hol} space on $X \setminus D$, \cite{Ko}, see also \S 1.1 in \cite{Au2}]
For $k \in \mathbb{Z}_{\geq 0}, \alpha \in (0,1)$, the {\hol} space $C^{k,\alpha} = C^{k,\alpha}(X \setminus D)$ is defined by the norm
\begin{equation}
\Vert f \Vert_{C^{k,\alpha}(X \setminus D)} := \Vert f \Vert_{C^{k,\alpha}(U_0)} + \max_{i=1,...,N}  \Vert  f \Vert _{C^{k, \alpha}(U_i \setminus D)},
\end{equation}
where the norm $\Vert  f \Vert _{C^{k, \alpha}(U_i \setminus D)}$ is defined by (\ref{local norm}).
\end{definition}
For a defining section $\sigma_D \in H^0 (X, L_X)$ of $D$ and a Hermitian metric $h_X$ on $L_X$, we set a function $t$ on $X \setminus D$ by
\begin{equation}
t := \log \Vert  \sigma_D \Vert _{h_X}^{-2}.
\end{equation}
We use $t$ as the weight function in order to define the weighted {\hol} space below.

\begin{definition}[The weighted {\hol} space on $X \setminus D$, \S 3.1 in \cite{Au2}]
\label{holder space}
For $\eta \in \mathbb{R}$, the weighted {\hol} space $C_{\eta}^{k,\alpha} = C_{\eta}^{k,\alpha} (X \setminus D)$ and its norm are defined by
\begin{eqnarray}
C_{\eta}^{k,\alpha} (X \setminus D) &:=& \{ f \in C^{k,\alpha}_{{\rm loc}} (X \setminus D) \hspace{3pt } | \hspace{3pt}  t^{-\eta} f \in C^{k,\alpha}(X \setminus D) \},\\
\Vert  f \Vert_{C_{\eta}^{k,\alpha} (X \setminus D)} &:=& \Vert t^{-\eta} f \Vert_{C^{k,\alpha} (X \setminus D)}. \label{weight norm}
\end{eqnarray}
\end{definition}
Note that if $f \in C_{\eta}^{k,\alpha}$ for $\eta < 0$, $f$ and its derivatives decay at infinity.

It is expected that the existence of cscK (more generally, extremal {\kah}) metrics of Poincar\'{e} type is equivalent to algebro-geometric stability as follows.

\begin{conj}[\S 3.1 in \cite{Sz1}, see also \cite{AAS, Se} for more precise statements]
$((X,L_X); D)$ has a cscK (extremal) {\kah} metric of Poincar\'{e} type iff it is (relative) K-polystable for angle 0.
\end{conj}

\subsection{{\kah} metrics with cone singularities}

In this subsection, we recall the definition of {\kah} metrics with cone singularities along $D$.
In addition, we recall the definition of cscK metric with cone singularities along $D$.
We take a real parameter $\beta \in (0,1]$.

\begin{definition}[\cite{Zhe1}]
A {\kah} metric $\omega$ on $X \setminus D$ is said to be \textit{a {\kah} cone metric of cone angle $2 \pi \beta$} if $\omega$ is quasi isometric to the flat cone metric:
\begin{equation}
\omega_{{\rm cone}} := \frac{\beta^2 \sqrt{-1} d z^1 \wedge d \overline{z^1}}{|z^1|^{2(1 - \beta)} } + \sum_{j=2}^{n} \sqrt{-1} d z^j \wedge d \overline{z^j},
\end{equation}
on the holomorphic coordinate chart $(U ; z^1,...,z^n)$ such that $D = \{ z^1 = 0\}$.
\end{definition}

We recall the definition of cscK metrics with cone singularities introduced by Zheng \cite{Zhe1}.
Let $\Omega$ be a {\kah} class and $\omega_{0} \in \Omega$ be a {\kah} metric.
We set
$$
c_1 (X , D, \beta) := c_1 (X) - (1-\beta) c_1 (L_D),
$$
where $L_D$ denotes the line bundle associated with $D$.
Fix a smooth form $\theta \in c_1 (X,D,\beta)$.
Let $s$ be a defining section of $D$ and $h$ be a Hermitian metric on $L_D$.
The symbol $\Theta_D$ denotes the curvature form multiplied by $\sqrt{-1}$, i.e., $\Theta_D = - \dol \log h$.
The $\partial \overline{\partial}$-lemma tells us that there exists a smooth function $f$ satisfying
$$
{\rm Ric}(\omega_{0}) = \theta + (1 - \beta)\Theta_D + \dol f.
$$
We can find the {\kah} cone metric $\omega_\theta$ of angle $2\pi\beta$ by solving the following singular complex {\ma} equation
$$
\omega_\theta^n = e^f \Vert  s \Vert ^{2\beta -2}_{h} \omega_{0}^n.
$$ 
Note that the {\kah} metric $\omega_\theta$ satisfies
$$
{\rm Ric}(\omega_\theta) = \theta +2\pi (1 - \beta) [D].
$$

\begin{definition}[Def 3.1 in \cite{Zhe1}]
\label{def cscK cone}
A {\it cscK cone metric} $\omega_{cscK}$ in the class $\Omega$ is the solution of the following coupled system in the sense of currents on $X$:
\begin{equation}
\label{coupled system}
\frac{\omega_{cscK}^{n}}{\omega_{\theta}^{n}} = e^F, \hspace{10pt} \Delta_{\omega_{cscK}} F = {\rm tr}_{\omega_{cscK}} \theta - \underline{S}_{\beta}.
\end{equation}
\end{definition}
Here, $\underline{S}_{\beta}$ denotes the topological constant defined by
\begin{equation}
\underline{S}_{\beta} := \frac{n c_1(X,D,\beta) \Omega^{n-1}}{\Omega^{n}}.
\end{equation}
We can easily check that the scalar curvature of cscK cone metric $\omega_{cscK}$ is equal to $\underline{S}_{\beta}$ on $X\setminus D$.
Conversely, we have
\begin{lemma}
\label{cscK cone lemma}
Let $\omega \in \Omega$ be a {\kah} cone metric of angle $2\pi\beta$.
Assume that the scalar curvature of $\omega$ is equal to $\underline{S}_{\beta}$, i.e., $S(\omega) = \underline{S}_{\beta}$ on $X\setminus D$.
Then, $\omega$ is a cscK cone metric, i.e., $\omega$ satisfies the coupled system (\ref{coupled system}) in Definition \ref{def cscK cone}.
\end{lemma}

\begin{proof}
We define a bounded function $F$ by
$$
F := \log \frac{\omega^n}{\omega_\theta^n}.
$$
Note that $F$ is smooth on $X \setminus D$ (in particular, $F \in C^{2,\alpha,\beta}$ (see Lemma 3.3 in \cite{Zhe1})).
On local holomorphic coordinates $(U; z^1 , ..., z^n)$ such that $D = \{ z^1 = 0 \}$, we can write $\Vert  s \Vert ^{2}_{h} = |z^1 |^2 e^{-a}$ for some smooth function $a$.
Note that $a$ satisfies $\Theta_D = - \dol a$
So, we have
\begin{eqnarray*}
\Delta_{\omega} F
&=& \Delta_{\omega} \log \frac{\omega^n}{\omega_\theta^n} \\
&=& \frac{ n\dol \log ( | z^1 |^{2(1-\beta)} \omega^n ) \wedge \omega^{n-1} }{\omega^{n}} - \frac{ n\dol \log (e^{f+(1 - \beta)a} \omega_{0}^{n} ) \wedge \omega^{n-1} }{\omega^{n}} \\
&=& -\underline{S}_{\beta} - {\rm tr}_{\omega} ( \dol f + (1-\beta) \Theta_D - {\rm Ric} \omega_{0} ) \\
&=& -\underline{S}_{\beta} + {\rm tr}_{\omega} \theta.\
\end{eqnarray*}
Note that the singularities of the volume form $\omega^n$ are canceled by multiplying the factor $| z^1 |^{2(1-\beta)}$.
Thus, this equation holds in the sense of currents on $X$.
\end{proof}

The existence of cscK cone metrics is related to algebro-geometric stability which is called log $K$-polystability.

\begin{conj}[Log Yau-Tian-Donaldson conjecture (see \cite{AHZ})]
The pair $((X,L_X); D)$ admits a cscK cone metric iff it is log $K$-polystable.
\end{conj}

The definitions of log $K$-polystability is given in \cite{Do2} (see also Section 3.1 of \cite{AHZ}).
The necessary condition of the existence of cscK cone metric is shown as follows.

\begin{thm}[\cite{AHZ}]
If $((X,L_X); D)$ admits a cscK cone metric of angle $2 \pi \beta$, then it is ${\rm Aut}_0 ((X,L_X); D)$-uniformly log $K$-stable with angle $2 \pi \beta$.
\end{thm}

It is also proved that the existence of cscK cone metric implies log K-(semi/poly)stability in \cite{AHZ}.
The definition of uniform log $K$-stability is also given in Section 3.1 of \cite{AHZ}.

\subsection{{\kah} cone metrics approximating a cscK metric of Poincar\'{e} type and fixed point formula}

We assume the existence of cscK metrics of Ponicar\'{e} type on $X \setminus D$ through out this paper.
In this subsection, we define a {\kah} metric with cone singularities along $D$ of cone angle $2 \pi \beta$, which converges to a cscK metric of Poincar\'{e} type when $\beta \to 0$.
In addition, we introduce a fixed point formula which characterizes a cscK cone metric.

We consider a polarized manifold $(X, L_X)$ with a smooth hypersurface $D \in | L_X |$.
Let $h_X$ be a Hermitian metric on $L_X$ such that its Chern curvature form multiplied by $\sqrt{-1}$ is a {\kah} form $\theta_X$, i.e., $\theta_X = - \dol \log h_X$.
For a defining section $\sigma_D \in H^0 (X, L_X)$ of $D$, we assume that $\Vert \sigma_D \Vert_{h_X}^{2} <e^{-2}$ by scaling.
We define a function $t(>2)$ on $X \setminus D$ by
\begin{equation}
t := \log \Vert  \sigma_D \Vert _{h_X}^{-2}.
\end{equation}

We set
$$
\underline{S} := \frac{n(c_1(X) - c_1 (L_X)) \cup c_1(L_{X})^{n-1}}{c_1(L_{X})^{n}}, \hspace{7pt} \underline{S}_{D}:= \frac{(n-1)(c_1(X) - c_1(L_X))|_D \cup (c_1(L_{X})|_D)^{n-2}}{(c_1(L_{X})|_D)^{n-1}}.
$$
Note that $\underline{S}$ is the average value of the scalar curvature of {\kah} metrics of Poincar\'{e} type on $X \setminus D$ in the class $[\theta_X]$ and $\underline{S}_{D}$ is the average value of the scalar curvature of {\kah} metrics on $D$ in the class $[\theta_X |_D]$.
Auvray showed the topological constraint for cscK metrics of Poincar\'{e} type in the class $[\theta_X]$.

\begin{thm}[\cite{Au1}]
\label{topological constraint}
If $X \setminus D$ admits a cscK metric of Ponicar\'{e} type, then the following inequality holds :
\begin{eqnarray}
\underline{S}_{D} > \underline{S}.
\end{eqnarray}
\end{thm}

\begin{rem}
\label{negative scalar curvature}
Since we assume that $D \in |L_X|$, we have $\underline{S} = \frac{n}{n-1}\underline{S}_{D}$.
So, Theorem \ref{topological constraint} implies the negativity of $\underline{S}_{D}$ under the assumption that there is a cscK metric of Poincar\'{e} type, i.e.,
$$
0 > \underline{S} - \underline{S}_{D} = \frac{n}{n-1}\underline{S}_{D} - \underline{S}_{D} = \frac{1}{n-1} \underline{S}_{D}.
$$
\end{rem}

By Theorem \ref{topological constraint}, we set the following well-defined positive number :
\begin{equation}
a_0 := \frac{2}{\underline{S}_{D} - \underline{S}} > 0.
\end{equation}
For sufficiently large $\lambda > 0$, we directly define a {\kah} metric of Poincar\'{e} type by
\begin{equation}
\theta_X -  a_0 \dol \log (\lambda + t) = \left( 1 - \frac{a_0}{\lambda + t} \right) \theta_X + \frac{a_0 \sqrt{-1} \partial t \wedge \overline{\partial} t}{(\lambda + t)^2}.
\end{equation}
We can compute as follows:
$$
\lambda + t = \lambda + \log \Vert \sigma_D \Vert_{h_X}^{-2} = \log \Vert e^{-\lambda/2} \sigma_D \Vert_{h_X}^{-2}.
$$
By replacing $\sigma_D$ with $e^{-\lambda/2} \sigma_D$, we write $\lambda + t$ as $t$ from now on.
By using this simple symbol $t$, we define a {\kah} metric of Poincar\'{e} type $\tilde{\omega}_{0}$ by
\begin{equation}
\tilde{\omega}_{0} := \theta_X -  a_0 \dol \log t.
\end{equation}

The reason why we put the coefficient $a_0$ is that the scalar curvature $S(\tilde{\omega}_{0})$ is asymptotic to the average value of the scalar curvature $\underline{S}$ (see Lemma \ref{estimate at infinity}).
The properties of cscK (more generally, extremal {\kah}) metrics of Poincar\'{e} type is well studied by Auvray as follows.

\begin{thm}[page 30 of \cite{Au3}]
\label{Auvray's decay}
If $\omega_{0}^{cscK} = \tilde{\omega}_{0} + \dol \varphi_{cscK}$ is a cscK metric of Poincar\'{e} type, then there exists a cscK metric $\theta_{D}^{cscK} = \theta_X |_D + \dol \psi_D$ on $D$.
Moreover, there is $\delta > 0$ such that
\begin{equation}
\varphi_{cscK} = p^* \psi_D + O(t^{-\delta})
\end{equation}
at any differential order near $D$, where $p(z^1, z^2,...,z^n) = (z^2,...,z^n)$ on $(U ; z^1,...,z^n)$ such that $D = \{ z^1 = 0 \}$.
\end{thm}

In this paper, we assume the non-existence of nontrivial holomorphic vector field on $D$, so a cscK metric on $D$ is unique in the {\kah} class $c_1 (L_X |_D)$ (\cite{Do}).
So, by replacing $\theta_X$ so that the restricted {\kah} metric $\theta_X |_D$ is a cscK metric, we can consider that the function $\varphi_{cscK}$ in Theorem \ref{Auvray's decay} decays near $D$ because we may assume that $\psi_D = 0$ from the uniqueness of cscK metrics on $D$.
Through out this paper, we fix the following notations.
\begin{definition}
\label{three conditions}
We assume that
\begin{equation}
\omega_{0}^{cscK} := \tilde{\omega}_{0} + \dol \varphi_{cscK} = \theta_X - a_0 \dol \log t + \dol \varphi_{cscK}
\end{equation}
is a cscK metric on $X\setminus D$ of Poincar\'{e} type.
By Theorem \ref{Auvray's decay},
\begin{itemize}
\item $\theta_X |_D$ is a cscK metric on $D$ and
\item the positive number $\delta > 0$ satisfies $\varphi_{cscK} = O(t^{-\delta})$ at any differential order near $D$.
\end{itemize}
\end{definition}

Next, we define a {\kah} metric with cone singularities approximating $\omega_{0}^{cscK}$.
Set
\begin{eqnarray}
\underline{S}_{\beta} &:=& \frac{nc_1(X,D,\beta) \cup c_1 (L_{X})^{n-1}}{c_1 (L_{X})^{n}},\\
\underline{S}_{D, \beta} &:=& \frac{(n-1)c_1(X,D,\beta)|_D \cup (c_1(L_{X})|_D)^{n-2}}{(c_1(L_{X})|_D)^{n-1}}.
\end{eqnarray}
Recall that $c_1 (X,D,\beta) = c_1 (X) - (1-\beta)c_1 (L_X)$ for the smooth divisor $D \in | L_X |$.
For sufficiently small $\beta>0$, we define the following well-defined number :
\begin{equation}
a_\beta := \frac{2}{\underline{S}_{D,\beta} - \underline{S}_{\beta}} > 0.
\end{equation}
Define functions $f_\beta (t)$ on $X \setminus D$ dy
\begin{equation}
f_\beta (t) := \frac{\beta a_\beta}{e^{\beta t} - 1}, \hspace{7pt} f_0 (t) := \frac{a_0}{t}.
\end{equation}
In addition, we set
\begin{equation}
G_\beta (t) := \int_{2}^{t} f_\beta (y) dy =\left\{
\begin{array}{ll}
a_\beta \log \left( (1- e^{-t \beta} )/(1 - e^{-2\beta}) \right) & (\beta > 0)\\
a_0 ( \log t - \log 2 ) & (\beta = 0).
\end{array}
\right.
\end{equation}
Since $(1 - e^{-\beta y})/ \beta \to y \hspace{5pt} (\beta \to 0)$, we have $G_\beta (t) \to a_0 ( \log t - \log 2 ) $ and $f_\beta (t) \to f_0 (t)$ as $\beta \to 0$ locally.
By the direct computation, we have
\begin{eqnarray}
\theta_X - \dol G_\beta (t)
&:=& \left( 1 - f_\beta (t) \right) \theta_X + \dot{f}_\beta (t) \sqrt{-1} \partial t \wedge \overline{\partial} t\\
&\hspace{-20pt}=&\hspace{-20pt} \left( 1 - \frac{\beta a_\beta}{e^{\beta t} - 1} \right) \theta_X +  \left( \frac{\beta}{1-e^{-\beta t}} \right)^2 a_\beta e^{-\beta t} \sqrt{-1} \partial t \wedge \overline{\partial} t. \label{last}
\end{eqnarray}
Note that $\theta_X - \dol G_\beta (t)$ is a {\kah} cone metric of cone angle $2 \pi \beta$ since the quadratic term $e^{-\beta t} \sqrt{-1} \partial t \wedge \overline{\partial} t$ included in (\ref{last}) is quasi isometric to $ \sqrt{-1} d z^1 \wedge d \overline{z^1}/|z^1|^{2(1 - \beta) }$ near $D = \{ z^1 = 0 \}$.
Here, we write $\Vert  \sigma_D \Vert ^{2}_{h_X} = |z^1 |^2 e^{-a}$ for some smooth function $a$.

\begin{definition}[The background {\kah} cone metric]
We define a {\kah} cone metric by
\begin{equation}
\omega_\beta := \theta_X - \dol G_\beta (t) + \dol \varphi_{cscK},
\end{equation}
where $\varphi_{cscK}$ is the function on $X \setminus D$ in Definition \ref{three conditions}.
\end{definition}

\begin{rem}
Note that $\omega_\beta$ is a {\kah} metric with cone singularities of angle $2\pi \beta$ along $D$ since the function $\varphi_{cscK}$ decays near $D$ (see Definition \ref{three conditions}) and does not affect the asymptotic behavior of the {\kah} cone metric $\theta_X - \dol G_\beta (t)$.
Moreover, $\omega_\beta$ converges locally to $\omega_{0}^{cscK}$ as $\beta \to 0$.
(In general, $\omega_\beta$ is not a cscK cone metric.)
\end{rem}

\begin{rem}
In \cite{Gu}, Guenancia uses the function $\log \left( (1 -  e^{- t\beta})/\beta \right) $ in order to define a {\kah} metric with cone singularities.
So, the function $G_\beta (t)$ is equal to the function in \cite{Gu} up to the coefficient $a_\beta$ (and additive constants).
The reason why we put coefficient $a_\beta$ is that we consider the construction of a cscK cone metric in the \textbf{fixed} cohomology class $c_1 (L_X)$ and the scalar curvature of the background cone metric $\omega_\beta$ is asymptotic to $\underline{S}_{\beta}$.
In \cite{Gu}, the cohomology class which contains a {\kah}-Einstein metric with cone singularities, 
is $c_1 (K_X) + (1-\beta)c_1 (D)$, so it depends on cone angle $2 \pi \beta$.
In this case, the average values $\underline{S}_{\beta}$ and $\underline{S}_{D, \beta}$ are equal to $n$ and $n-1$ respectively and the corresponding coefficient is $2$ (independent of $\beta$  !).
Thus, this problem of the setting of coefficients does not arise in \cite{Gu}.
\end{rem}

Finally, we consider the fixed point formula on $C^{4,\alpha}_{\eta}$ which characterizes a cscK cone metric.
For a function $\phi_\beta \in C^{4,\alpha}_{\eta} , \eta <0$, we consider the following expansion:
\begin{equation}
\label{fixed point}
S (\omega_\beta + \dol \phi_\beta) = S (\omega_\beta) + L_{\omega_\beta}(\phi_\beta)  + Q_{\omega_\beta}(\phi_\beta).
\end{equation}
Here, $L_{\omega_\beta} : C^{4,\alpha}_{\eta} \to C^{0,\alpha}_{\eta}$ is the lineariztion of the scalar curvature operator and $Q_{\omega_\beta}$ is the remaining nonlinear term.
We characterize a solution of the equation:
\begin{equation}
\label{cscK 1}
S (\omega_\beta + \dol \phi_\beta) = \underline{S}_\beta, \hspace{10pt}  \phi_\beta \in C^{4,\alpha}_{\eta},
\end{equation}
as a fixed point $\phi_\beta \in C^{4,\alpha}_{\eta}$ given by
\begin{equation}
\phi_\beta = - L_{\omega_\beta}^{-1} (S (\omega_\beta) - \underline{S}_\beta + Q_{\omega_\beta}(\phi_\beta)).
\end{equation}
Note that $\omega_\beta + \dol \phi_\beta$ in (\ref{cscK 1}) is a cscK cone metric by Lemma \ref{cscK cone lemma}.

\begin{rem}
For a {\kah} cone metric of angle $2 \pi \beta$, there are the function spaces denoted by $C^{k,\alpha,\beta}$ for $k=0,1,...,4$, defined by Donaldson \cite{Do2} and Li-Zheng \cite{LZ}.
By considering these function space, we can show the existence of a cscK cone metric on $X \setminus D$ under the assumption that $D$ is of sufficiently large degree (\cite{AHZ}).
In our case, we consider the functional space $C^{k,\alpha}_{\eta}$ since we want to use Sektnan's result \cite{Se} later.
\end{rem}


\section{cscK cone metrics and convergence}

In this section, we prove Theorem \ref{cone to cusp} by constructing a fixed point in $C^{4,\alpha}_{\eta}$, which characterizes a cscK cone metric (\ref{fixed point}).
From now on, we fix the exponent $\alpha \in (0,1)$.
In order to show that, we study the following.\\
1. The weighted {\hol} norm of the scalar curvature of the {\kah} cone metric $\omega_\beta$, i.e., $\Vert  S(\omega_\beta) - \underline{S}_{\beta} \Vert _{C^{0,\alpha}_{\eta}}$ can be made small arbitrarily as $\beta \to 0$.
Here, the weight $\eta $ is independent of $\beta$.\\
2. The linearization of the scalar curvature operator has a bounded inverse, i.e., there is $L_{\omega_\beta}^{-1}$ and its operator norm is independent of cone angle $2 \pi \beta$.

\subsection{The estimate of the scalar curvature}

In this subsection, we prove the following estimate.

\begin{prop}
\label{the norm of scalar curvature}
For a weight $\eta \in ( - \delta,0)$, we have
\begin{equation}
\label{global estimate}
\Vert  S(\omega_\beta) - \underline{S}_{\beta} \Vert _{C^{0,\alpha}_{\eta}(X \setminus D)} = O((- \beta \log \beta)^{\eta+\delta}).
\end{equation}
Here, the number $\delta$ is given in Definition \ref{three conditions}.
\end{prop}

In order to prove this proposition, we divide $X \setminus D$ into a neighborhood of $D$ and a region away from $D$.
Firstly, we study the asymptotic behavior of the scalar curvature of the {\kah} metric $\theta_X - \dol G_\beta (t)$ on a neighborhood of $D$.

\begin{lemma}
\label{first asymptotic}
If $\theta_D := \theta_X |_D$ is a cscK metric on $D$, we have
\begin{equation}
S( \theta_X - \dol G_\beta (t) ) - \underline{S}_\beta = O(\Vert  \sigma_D \Vert ^{2(1-\beta)}) = O(e^{( \beta - 1)t})
\end{equation}
as $t \to \infty$ $($equivalently, $\sigma_D \to 0)$.
\end{lemma}

\begin{proof}
From (\ref{last}), the volume form of $ \theta_X - \dol G_\beta (t) $ can be written as
\begin{eqnarray}
(\theta_X - \dol G_\beta (t))^n
&=&  ( 1 + O(\Vert  \sigma_D \Vert ^2) )   \left( \frac{\beta}{1-e^{-\beta t}} \right)^2 n a_\beta e^{-\beta t} \theta_{X}^{n-1} \wedge \partial t \wedge \overline{\partial} t \
\end{eqnarray}

So, the Ricci form of this {\kah} metric is
\begin{eqnarray*}
&&{\rm Ric} (\theta_X - \dol G_\beta (t))\\
&=& - \dol \log ( 1 + O(\Vert  \sigma_D \Vert ^2) )  + 2 \dol \log (1 - e^{- \beta t}) + \beta \theta_X + {\rm Ric}\theta_X \\
&=& - \dol \log ( 1 + O(\Vert  \sigma_D \Vert ^2) )  + \dol \left( (\underline{S}_{\beta} - \underline{S}_{D,\beta} ) G_\beta (t) \right) + \beta \theta_X + {\rm Ric}\theta_X. \
\end{eqnarray*}
In order to obtain the following three estimates;
\begin{equation}
\label{estimate 1}
{\rm tr}_{\theta_X - \dol G_\beta (t)} \dol \log ( 1 + O(\Vert  \sigma_D \Vert ^2)  = O(\Vert  \sigma_D \Vert ^{2(1-\beta)} ),
\end{equation}
\begin{equation}
\label{estimate 2}
{\rm tr}_{\theta_X - \dol G_\beta (t)} \dol \left( (\underline{S}_{\beta} - \underline{S}_{D,\beta} ) G_\beta (t) \right) = \underline{S}_{\beta} - \underline{S}_{D,\beta} + O(\Vert  \sigma_D \Vert ^{2(1-\beta)} ),
\end{equation}
\begin{equation}
\label{estimate 3}
{\rm tr}_{\theta_X - \dol G_\beta (t)} ( \beta  \theta_X + {\rm Ric}\theta_X ) = \beta(n-1) + S(\theta_D) + O(\Vert  \sigma_D \Vert ^{2(1-\beta)}),
\end{equation}
we use the following linear algebraic result (see \cite[p.24]{Zhan} and \cite[\S 3.1 and \S 3.2]{Ao}).
We consider the following matrix
\begin{eqnarray*}
T = \left[ 
\begin{array}{cc}
A & B \\
C & D \\
\end{array} 
\right]
\end{eqnarray*}
and assume that $D$ and $S := A - B D^{-1} C$ are invertible.\
Then, $T$ is invertible and the inverse matrix of $T$ is given by
\begin{eqnarray}
\label{inverse matrix}
T^{-1} = \left[ 
\begin{array}{cc}
S^{-1} & - S^{-1} B D^{-1} \\
- D^{-1} C S^{-1} &  D^{-1} + D^{-1} C S^{-1} B D^{-1}\\
\end{array} 
\right].
\end{eqnarray}
By taking normal holomorphic coordinates in \cite[Prop 5]{Ao}, we can write as
\begin{eqnarray*}
\theta_X - \dol G_\beta (t) = \left[ 
\begin{array}{cc}
\frac{a_\beta}{|z^1|^{2(1 - \beta)}} & 0 \\
0 &  g_{i \overline{j}}  \\
\end{array} 
\right] + O(|z^1|^2)
\end{eqnarray*}
near $D$, where we write $\theta_X = i g_{i \overline{j}} dz^i \wedge d \overline{z}^j$.
By applying the formula (\ref{inverse matrix}), we have
\begin{eqnarray*}
\left( \theta_X - \dol G_\beta (t) \right)^{-1} = \left[ 
\begin{array}{cc}
\frac{|z^1|^{2(1 - \beta)}}{a_\beta} & 0 \\
0 &  g^{i \overline{j}}  \\
\end{array} 
\right] + O(|z^1|^{2(1-\beta)}).
\end{eqnarray*}
Here, $(g^{i \overline{j}}) = (g_{i \overline{j}})^{-1}$.
Thus, we obtain directly the estimate (\ref{estimate 1}) and (\ref{estimate 2}).

Since the scalar curvature $S(\theta_D)$ is the trace of the Ricci form ${\rm Ric} (\theta_D)$, we can obtain similarly the estimate (\ref{estimate 3}).
Note that $\theta_D$ is a cscK metric on $D$, i.e., $S(\theta_D ) = \underline{S}_{D}$.

Therefore, we have
$$
 \underline{S}_{\beta} - \underline{S}_{D,\beta} + \beta(n-1) + \underline{S}_{D} = \underline{S}_{\beta}.
$$
Thus, we have finished the proof of this lemma.
\end{proof}

Recall the definition of $\omega_\beta$ :
$$
\omega_\beta := \theta_X - \dol G_\beta (t) + \dol \varphi_{cscK}.
$$
Since $S(\omega_\beta) = S(\theta_X - \dol G_\beta (t) ) + L_{\omega_\beta} ( \varphi_{cscK}) + Q_{\omega_\beta} ( \varphi_{cscK})$ and $\varphi_{cscK} = O(t^{-\delta})$ near $D$, we have the following.
\begin{lemma}
\label{estimate at infinity}
There is $\delta > 0$ such that
\begin{equation}
S(\omega_\beta) - \underline{S}_{\beta} = O(t^{-\delta})
\end{equation}
at any differential order near $D$.
\end{lemma}

Secondly, we study the estimate of $S(\theta_X - \dol G_\beta (t))$ away from $D$.
We show the following.

\begin{prop}
\label{le}
On $V_\beta := \{ t \leq (- \beta \log \beta)^{-1}  \} \Subset X \setminus D$, we have
\begin{equation}
S(\theta_X - \dol G_\beta (t) ) - S(\tilde{\omega}_{0}) = O (\beta)
\end{equation}
\end{prop}

\begin{proof}
Note that
$$
\theta_X - \dol G_\beta (t) = (1 - f_\beta (t) )\theta_X + \dot{f}_\beta (t) \sqrt{-1} \partial t \wedge \overline{\partial} t
$$
and
$$
\tilde{\omega}_{0} = \theta_X - \dol G_0 (t) = (1 - f_0 (t) )\theta_X + \dot{f}_0 (t) \sqrt{-1} \partial t \wedge \overline{\partial} t.
$$
Their volume forms are given by
\begin{equation}
(\theta_X - \dol G_\beta (t))^{n} = \left( 1 - f_\beta (t) \right)^{n-1}  \left( 1 - f_\beta (t)  + \dot{f}_\beta \Vert  \partial t \Vert _{\theta_X}^{2} \right) \theta_{X}^{n}
\end{equation}
and
\begin{equation}
\tilde{\omega}_{0}^{n} = \left( 1 - f_0 (t)  \right)^{n-1}  \left( 1 - f_0 (t)  + \dot{f}_0 \Vert  \partial t \Vert _{\theta_X}^{2} \right) \theta_{X}^{n}
\end{equation}
respectively.
So, the Ricci forms of these {\kah} metrics are given by
\begin{eqnarray}
\nonumber
{\rm Ric} (\theta_X - \dol G_\beta (t))&=& {\rm Ric} (\theta_X) - (n-1) \dol \log \left( 1 - f_\beta (t) \right)\\
 &-& \dol \log \left( 1 - f_\beta (t)  + \dot{f_\beta} (t) \Vert  \partial t \Vert _{\theta_X}^{2} \right).
\end{eqnarray}
and
\begin{eqnarray}
{\rm Ric} (\tilde{\omega}_{0})\nonumber
&=& {\rm Ric} (\theta_X) - (n-1) \dol \log \left( 1 - f_0 (t) \right)\\
 &-& \dol \log \left( 1 - f_0 (t)  + \dot{f_0} (t) \Vert  \partial t \Vert _{\theta_X}^{2} \right).
\end{eqnarray}
Thus, in order to study the difference of their scalar curvature, it suffices to show that any derivatives of $f_\beta$ are uniformly continuous and converges uniformly to the derivatives of $f_0$ on $V_\beta$.
Since $t > 2$, we have
$$
f_\beta (t) < 1/2.
$$
Recall that 
$$
f_\beta (t) = \frac{\beta a_\beta}{ e^{ \beta t} - 1}.
$$
Directly, we have
$$
\dot{f_\beta} (t) = -  f_\beta (t)^2 e^{ \beta t}/ a_\beta.
$$
Thus, the family of functions $\{ f_\beta \}_\beta$ is uniformly continuous.
Inductively, we obtain
$$
\ddot{f_\beta} (t) = - \beta f_\beta (t)^2 e^{ \beta t}/ a_\beta + 2 f_\beta (t)^3 e^{ 2 \beta t}/ a_{\beta}^{2},
$$
so any differential of $f_\beta$ is uniformly continuous.

To show the uniform convergence of the family $\{ f_\beta \}$, we use the following elementary inequalities :
\begin{equation}
0 \geq \frac{\beta}{e^{ \beta t} - 1} - \frac{1}{t} \geq \frac{  e^{-\beta t} - 1}{t} \geq - \beta e^{\beta t}.
\end{equation}
On $V_\beta$, there exists a constant $C>0$ independent of $\beta$ such that
\begin{equation}
|f_0 (t) - f_\beta (t)| < C \beta.
\end{equation}
This estimate implies that for any $k$, there exists $C_k >0$ independent of $\beta>0$ such that
$$
| f^{(k)}_{\beta} (t) - f^{(k)}_{0} (t) | <C_k \beta \hspace{10pt} {\rm on} \hspace{5pt} V_\beta .
$$
So, we have finished proving this proposition.
\end{proof}

Recall the definitions of $\omega_{0}^{cscK}$ and $\omega_{\beta}$ :
$$
\omega_{0}^{cscK} := \theta_X - \dol G_0 (t) +\dol \varphi_{cscK}, \hspace{10pt} \omega_{\beta} := \theta_X - \dol G_\beta (t) + \dol \varphi_{cscK} .
$$
So, the difference of the scalar curvatures $S(\omega_{0}^{cscK})$ and $S(\omega_\beta)$ comes from the function $G_\beta (t)$ and its derivatives.
By applying Proposition \ref{le} to these metrics, we have
\begin{lemma}
\label{local estimate}
On $V_\beta$, we have
\begin{equation}
\Vert  S (\omega_{0}^{cscK} )- S(\omega_\beta) \Vert _{C^{0,\alpha}_{\eta} (V_\beta)} \leq C \beta.
\end{equation}
for some $C>0$ independent of $\beta$.
\end{lemma}

\hspace{-20pt} {\it Proof of Proposition \ref{the norm of scalar curvature}}

On $V_\beta$, we can write as follows :
$$
S(\omega_\beta) - \underline{S}_{\beta} = S(\omega_\beta) - S(\omega_{0}^{cscK}) +  S(\omega_{0}^{cscK}) - \underline{S}_{\beta} .
$$
Note that we have $S(\omega_{0}^{cscK}) - \underline{S}_{\beta} = - \beta n$.
Lemma \ref{local estimate} gives the estimate (\ref{global estimate}) on $V_\beta$.

Since $\omega_\beta = \theta_X - \dol G_\beta (t)  + \dol \varphi_{cscK}$ on the complement of $V_\beta$,
Lemma \ref{estimate at infinity} implies the estimate (\ref{global estimate}). \sq

\subsection{The bounded inverse of the linearization}

For a {\kah} metric $\omega$, the linearization of the scalar curvature operator $L_\omega$ satisfies
$$
L_\omega = - \mathcal{D}_{\omega}^{*} \mathcal{D}_{\omega} + (\nabla^{1,0} S (\omega) , \nabla^{1,0} *)_{\omega}.
$$
Here, $\mathcal{D}_{\omega} = \overline{\partial} \circ \nabla^{1,0} $ and $\nabla^{1,0}$ denotes the (1,0)-gradient with respect to $\omega$.
We call $\mathcal{D}_{\omega}^{*} \mathcal{D}_{\omega}$ {\it the Lichnerowicz operator.}
If $\mathcal{D}_{\omega} \phi = 0$, $\nabla^{1,0} \phi$ is a holomorphic vector field.

In this subsection, we show the uniform lower estimate of the Lichnerowicz operator $\mathcal{D}_{\omega_\beta}^{*} \mathcal{D}_{\omega_\beta} : C^{4, \alpha}_{\eta}(X \setminus D) \to C^{0, \alpha}_{\eta}(X \setminus D)$.
\begin{thm}
\label{uniform lemma}
Assume that ${\rm Aut}_0 ((X,L_X); D)$ is trivial and $H^0 (D, TD) = 0$.
Then, there exist $\kappa > 0$ and $\beta_0 > 0$ which satisfy the following property.
For $\eta \in (-\kappa , 0)$, there exists $K > 0$ independent of $\beta \in [0, \beta_0)$ such that
$$
\Vert  \mathcal{D}_{\omega_\beta}^{*} \mathcal{D}_{\omega_\beta} \phi \Vert _{C^{0, \alpha}_{\eta}} \geq K \Vert \phi \Vert _{C^{4, \alpha}_{\eta}}, \hspace{7pt} \forall \phi \in C^{4, \alpha}_{\eta}.
$$
\end{thm}
\begin{rem}
The constant $K>0$ in Theorem \ref{uniform lemma} depends on $\alpha, \kappa, \eta$ and $\delta$.
\end{rem}

In order to study the inverse operator of the Lichnerowicz operator $\mathcal{D}_{\omega_\beta}^{*} \mathcal{D}_{\omega_\beta} : C^{4, \alpha}_{\eta}(X \setminus D) \to C^{0, \alpha}_{\eta}(X \setminus D)$, we use the following proposition in \cite{Se} for the Lichnerowicz operator $\mathcal{D}_{\omega_{0}^{cscK}}^{*} \mathcal{D}_{\omega_{0}^{cscK}} : C^{4, \alpha}_{\eta}(X \setminus D) \to C^{0, \alpha}_{\eta}(X \setminus D)$ when $\beta =0$.

\begin{prop}[Prop 4.3 in \cite{Se} when $r =h^0 (D, TD) = 0$ and ${\rm Aut}_0 ((X,L_X); D)$ is trivial]
\label{Sektnan}
Assume that $H^0 (D, TD) = 0$ and ${\rm Aut}_0 ((X,L_X); D)$ is trivial.
There is $\kappa >0$ with the following properties:
For $\eta \in (-\kappa , 0)$, we have
\begin{itemize}
\item ${\rm Ker}  ( \mathcal{D}_{\omega_{0}^{cscK}}^{*} \mathcal{D}_{\omega_{0}^{cscK}} : C^{4, \alpha}_{ \eta} \to C^{0, \alpha}_{ \eta}) = 0$,
\item ${\rm Im} ( \mathcal{D}_{\omega_{0}^{cscK}}^{*} \mathcal{D}_{\omega_{0}^{cscK}} : C^{4, \alpha}_{\eta} \to C^{0, \alpha}_{\eta}) = C^{0, \alpha}_{\eta}$.
\end{itemize}
\end{prop}
After this, we fix the symbol $\kappa$ satisfying Theorem \ref{Sektnan}.
Thus if $\eta \in (-\kappa,0)$ and $H^0 (D, TD) = 0$ and ${\rm Aut}_0 ((X,L_X); D)$ is trivial, the Lichnerowicz operator $\mathcal{D}_{\omega_{0}^{cscK}}^{*} \mathcal{D}_{\omega_{0}^{cscK}} : C^{4, \alpha}_{\eta} \to C^{0, \alpha}_{\eta})$ is isomorphic.

\begin{rem}
Sektnan showed more general result for the case when $X$ and $D$ admit nontrivial holomorphic vector fields (see Proposition 4.3 in \cite{Se}).
In our case, we assume that $H^0 (D,TD)=0$ and ${\rm Aut}_0 ((X,L_X); D)$ is trivial since we want to deal with the case that the Lichnerowicz operator is isomorphic.
\end{rem}

\begin{lemma}
\label{continuity}
The following map is continuous with respect to the operator norm:
$$
[0,1] \ni \beta \mapsto \mathcal{D}_{\omega_\beta}^{*} \mathcal{D}_{\omega_\beta} \in {\rm Map}( C^{4, \alpha}_{\eta}(X \setminus D) \to C^{0, \alpha}_{\eta}(X \setminus D)).
$$
\end{lemma}

\begin{proof}
The Lichnerowicz operator can be written locally as $\mathcal{D}_{\omega_\beta}^{*} \mathcal{D}_{\omega_\beta} = g_{\beta}^{i \overline{j}} g_{\beta}^{k \overline{l}} \nabla_{i} \nabla_{j} \nabla_{\overline{k}} \nabla_{\overline{l}}$, where we write $\omega_\beta = \sqrt{-1} g_{\beta, i \overline{j}} dz^i \wedge d\overline{z}^j$ and $(g_{\beta}^{i \overline{j}}) = (g_{\beta, i \overline{j}})^{-1}$ (see \S 4.1 in \cite{Sz2}).
Since the operator norm of $\mathcal{D}_{\omega_\beta}^{*} \mathcal{D}_{\omega_\beta}$ is defined by
\begin{eqnarray}
\Vert  \mathcal{D}_{\omega_\beta}^{*} \mathcal{D}_{\omega_\beta} \Vert _{C^{4, \alpha}_{\eta} \to C^{0, \alpha}_{\eta}} &:=& \sup \{ \Vert  \mathcal{D}_{\omega_\beta}^{*} \mathcal{D}_{\omega_\beta} \phi \Vert _{C^{0, \alpha}_{\eta}} \hspace{3pt} |  \hspace{3pt} \Vert   \phi \Vert _{C^{4, \alpha}_{\eta}} \leq 1 \},
\end{eqnarray}
it suffices to show that
\begin{equation}
 \Vert  g_{\beta}^{i \overline{j}} - g_{0}^{i \overline{j}} \Vert _{C^{0, \alpha}_{0}} \to 0
\end{equation}
and any derivatives of $g_{\beta, i \overline{j}}$ converge to $g_{0, i \overline{j}}$ as $\beta \to 0$.
This convergence follows from the fact that
$$
\omega_{\beta} = \left( 1 - f_\beta (t) \right) \theta_X + \dot{f}_\beta (t) \sqrt{-1} \partial t \wedge \overline{\partial} t + \dol \varphi_{cscK}
$$
and any derivatives $f_{\beta}^{(k)}$ converges to $f_{0}^{(k)}$ uniformly when $\beta \to 0$.
\end{proof}

\hspace{-20pt} {\it  Proof of Theorem \ref{uniform lemma}}

Fix $\eta \in (-\kappa , 0)$.
By Proposition \ref{Sektnan}, we know that $ \mathcal{D}_{\omega_\beta}^{*} \mathcal{D}_{\omega_\beta}$ is isomorphic at $\beta = 0$.
It follows from the closed mapping theorem (see page 77 of \cite{Yo}) that there is $K_0$ satisfying 
$$\Vert  \mathcal{D}_{\omega_{0}^{cscK}}^{*} \mathcal{D}_{\omega_{0}^{cscK}} \phi \Vert _{C^{0, \alpha}_{\eta}} \geq K_0 \Vert \phi \Vert _{C^{4, \alpha}_{\eta}}, \hspace{7pt} \forall \phi \in C^{4, \alpha}_{\eta}.$$
Set $\beta_0 := \sup \{ \hspace{3pt}  \beta \geq 0  \hspace{3pt}  |  \hspace{3pt}  \exists  \hspace{3pt}  \mathcal{D}_{\omega_\beta}^{*} \mathcal{D}_{\omega_\beta}^{-1} : C^{0, \alpha}_{\eta} \to C^{4, \alpha}_{\eta} \hspace{3pt}  \}$.
From Lemma \ref{continuity}, we have $\beta_0 > 0$.  \sq

\subsection{Proof of Theorem \ref{cone to cusp}}

In this subsection, we show the map $\mathcal{N}_\beta : C^{4, \alpha}_{\eta} \to C^{4, \alpha}_{\eta}$ defined by
\begin{equation}
\mathcal{N}_\beta ( \phi_\beta ) = - L_{\omega_\beta}^{-1}( S(\omega_\beta) - \underline{S}_{\beta} + Q_{\omega_\beta} ( \phi_\beta) )
\end{equation}
has a fixed point, i.e., there exists a cscK cone metric of angle $2\pi\beta$.
In order to show this, we prove that the map $\mathcal{N}_\beta : C^{4, \alpha}_{\eta} \to C^{4, \alpha}_{\eta}$ is a contraction map.

\begin{rem}
This proof is an application of the construction of a cscK metric on the blowing-up of a cscK manifold without nontrivial holomorphic vector field by Arezzo-Pacard \cite{AP1,AP2}.
\end{rem}

Recall that $L_{\omega_{\beta}} = - \mathcal{D}_{\omega_\beta}^{*} \mathcal{D}_{\omega_\beta} + (\nabla^{1,0} S (\omega_\beta) , \nabla^{1,0} *)_{\omega_\beta}$.
Since $S(\omega_\beta)$ converges to the constant $\underline{S}$ as $\beta \to 0$, Theorem \ref{uniform lemma} implies
\begin{thm}
\label{inverse}
Assume that ${\rm Aut}_0 ((X,L_X); D)$ is trivial and $H^0 (D, TD) = 0$.
Then, there exist $\kappa > 0$ and $\beta_0 > 0$ which satisfy the following property.
For $\eta \in (-\kappa , 0)$, there exists $\hat{K} > 0$ independent of $\beta \in [0, \beta_0)$ such that
$$
\Vert  L_{\omega_\beta} \phi \Vert _{C^{0, \alpha}_{\eta}} \geq \hat{K} \Vert \phi \Vert _{C^{4, \alpha}_{\eta}}, \hspace{7pt} \forall \phi \in C^{4, \alpha}_{\eta}.
$$
Namely, the operator norm of $L_{\omega_\beta}^{-1} : C^{0, \alpha}_{\eta} \to C^{4, \alpha}_{\eta}$ is bounded by $1/\hat{K}$.
\end{thm}
\begin{rem}
The constant $\hat{K}>0$ in Theorem \ref{inverse} depends on $\alpha, \kappa, \eta$ and $\delta$.
\end{rem}
We need the following lemma.
\begin{lemma}
\label{difference}
There exists $c_0 > 0$ independent of sufficiently small $\beta > 0$ such that if $\Vert  \phi \Vert _{C^{4, \alpha}_{\eta}} \leq c_0$, we have
\begin{equation}
\Vert  L_{\omega_\beta + \dol \phi} - L_{\omega_\beta} \Vert _{C^{4, \alpha}_{\eta} \to C^{0, \alpha}_{\eta}} \leq \hat{K}/2
\end{equation}
and $\omega_\beta + \dol \phi$ is positive definite.
\end{lemma}

\begin{proof}
Let us write locally as $\omega_\beta = \sqrt{-1} g_{i \overline{j}} d z^i \wedge d \overline{z^j} $.
The Lemma follows from the simple relation;
$$
g_{\phi}^{-1} - g^{-1} = - g_{\phi}^{-1} \left( g_{\phi} - g \right) g^{-1},
$$
because the weight $\eta$ is negative.
\end{proof}

By using Theorem \ref{inverse} and Lemma \ref{difference}, we can show that $\mathcal{N}_{\beta}$ is a contraction map.

\begin{lemma}
\label{contraction}
Let $c_0$ be the constant given in Lemma \ref{difference}.
If $\Vert  \phi \Vert _{C^{4, \alpha}_{\eta}}, \Vert  \psi \Vert _{C^{4, \alpha}_{\eta}} \leq c_0$, we have
\begin{equation}
\Vert  \mathcal{N}_\beta ( \phi ) - \mathcal{N}_\beta ( \psi)  \Vert _{C^{4, \alpha}_{\eta}} \leq \frac{1}{2} \Vert  \phi - \psi \Vert _{C^{4, \alpha}_{\eta}}.
\end{equation}
\end{lemma}
\begin{proof}

Since the operator $\mathcal{N}_\beta$ is defined by $\mathcal{N}_\beta(\phi) := - L_{\omega_{\beta}}^{-1}( S(\omega_{\beta})- \underline{S}_{\beta} + Q_{\omega_{\beta}}(\phi) )$, we have
$$
\mathcal{N}_{\beta}(\phi) - \mathcal{N}_{\beta}(\psi) = - L_{\omega_{\beta}}^{-1}( Q_{\omega_{\beta}}(\phi) - Q_{\omega_{\beta}}(\psi) ).
$$
By the mean value theorem, there exists $\chi = t \phi + (1 - t) \psi$ for $t \in [0,1]$ such that
$$
DQ_{\omega_{\beta},\chi} (\phi - \psi) = Q_{\omega_{\beta}}(\phi) - Q_{\omega_{\beta}}(\psi).
$$
Here, $DQ_{\omega_{\beta},\chi}$ denotes the derivative of $Q_{\omega_{\beta}}$ at $\chi$.
We write $\omega_\chi = \omega_\beta + \dol \chi$.
By differentiating the equation $S(\omega_\beta) + L_{\omega_\beta}(\chi + s f) + Q_{\omega_\beta} (\chi + sf)  = S(\omega_\chi + s \dol f ) = S(\omega_\chi) + s L_{\omega_\chi} f + Q_{\omega_\chi} (sf)$ at $s=0$, we obtain
$$
DQ_{\omega_{\beta},\chi} = L_{\omega_{\chi}} - L_{\omega_{\beta}}.
$$
We can easily show that  $||\chi||_{C^{4,\alpha}_{\eta}}  \leq c_{0}$.
By Theorem \ref{inverse} and Lemma \ref{difference}, we finish the proof.

\end{proof}

\hspace{-20pt} {\it Proof of Theorem \ref{cone to cusp}}

We take $\alpha \in (0,1), \delta, \kappa$ and $\eta$ as before (see Definition \ref{three conditions} and Proposition \ref{Sektnan}).
Fix a small positive number $\epsilon > 0$ so that $\delta + \eta > \epsilon$.
\begin{equation}
\mathscr{U}_\beta := \{ \phi \in C^{4, \alpha}_{\eta} \hspace{3pt} | \hspace{3pt} \Vert  \phi \Vert _{C^{4, \alpha}_{\eta}} \leq c_0 (- \log \beta )^{- \epsilon} \}
\end{equation}
Lemma \ref{contraction} implies that the map $\mathcal{N}_\beta$ is a contraction map on $\mathscr{U}_\beta$.
From (\ref{global estimate}), we can find a sufficiently small $\beta > 0$ so that
\begin{equation}
\Vert  S(\omega_\beta) - \underline{S}_{\beta} \Vert _{C^{0,\alpha}_{\eta}} \leq c_0 \hat{K} (-\log \beta)^{-\epsilon}/ 2.
\end{equation}
For $\phi \in \mathscr{U}_\beta$, we have
\begin{eqnarray}
\Vert  \mathcal{N}_\beta (\phi) \Vert _{C^{4,\alpha}_{\eta}} &\leq& \Vert  \mathcal{N}_\beta (\phi) - \mathcal{N}_\beta (0) \Vert _{C^{4,\alpha}_{\eta}} + \Vert  L_{\omega_\beta}^{-1} ( S (\omega_\beta ) -\underline{S}_{\beta} )\Vert _{C^{4,\alpha}_{\eta}}\\
&\leq& \frac{1}{2} c_0 (- \log \beta )^{- \epsilon} + \hat{K}^{-1} \Vert  S (\omega_\beta ) -\underline{S}_{\beta}\Vert _{C^{0,\alpha}_{\eta}} \\
&\leq& c_0 (- \log \beta )^{- \epsilon}.
\end{eqnarray}
So, we note that $\mathcal{N}_\beta(\mathscr{U}_\beta) \subset \mathscr{U}_\beta$.
Therefore, we can find $\phi_\beta \in \mathscr{U}_\beta$ so that $\omega_\beta + \dol \phi_\beta$ is a cscK cone metric for sufficiently small cone angle $2\pi \beta$.
Moreover, by definition of $\mathscr{U}_\beta$, the sequence of cscK cone metrics $\omega_\beta + \dol \phi_\beta$ locally converges to $\omega_{0}^{cscK}$.


\bigskip
\address{
Osaka Prefectural Abuno High School,\\
3-38-1, Himuro-chou, Takatsuki-shi,\\
Osaka, 569-1141\\
Japan
}
{takahiro.aoi.math@gmail.com}

\end{document}